\documentclass[a4paper,, 11pt]{amsart}
\usepackage{amsmath}
\usepackage{amsthm}
\usepackage{amssymb}
\usepackage{latexsym}
\usepackage{amscd}

\newtheorem{theorem}{Theorem}[section]

\newtheorem{lemma}[theorem]{Lemma}
\newtheorem{proposition}[theorem]{Proposition}
\newtheorem{claim}[theorem]{Claim}
\newtheorem{corollary}[theorem]{Corollary}
\theoremstyle{definition}
\newtheorem{notation}[theorem]{Notation}
\newtheorem{definition}[theorem]{Definition}
\newtheorem{example}[theorem]{Example}

\theoremstyle{remark}
\newtheorem{remark}[theorem]{Remark}

\numberwithin{equation}{section}

\begin{document}
\baselineskip = 1,5\baselineskip
\title{Spaces of $\mathbb R$ - places of rational function fields }

\author{Micha\l\; Machura\, \,and \,\, Katarzyna Osiak}
\address{Institute of Mathematics, Silesian University, Bankowa 14, 40-007 Katowice, Poland}
\email{kosiak@ux2.math.us.edu.pl\\machura@ux2.math.us.edu.pl}



\begin{abstract}
In the paper  an answer to a problem "When different orders of R(X) (where R is a
real closed field) lead to the same real place ?"  is given. We use this result to show that the space
of $\mathbb R$-places of the field $\textbf{R}(Y)$ (where \textbf{R} is any real closure of $\mathbb R(X)$) is not metrizable space. Thus the space $M(\mathbb R(X,Y))$ is not metrizable, too.
\end{abstract}

\maketitle

\section{Introduction}
The studies on  real places of formally real fields were initiated by  Dubois \cite{du} and Brown \cite{br}.
The research has been continued in several papers: Brown and Marshall \cite{bm}, Harman \cite{h}, Sch\"{u}lting \cite{sch}, Becker and Gondard \cite{bg} and Gondard and Marshall \cite{gm}. We will use notation and  terminology introduced by Lam \cite {l} - most of the results we shall recall in this section can be found there.
We assume that the reader is somewhat familiar with the valuation theory and the theory of formally real (ordered) fields.

Let $K$ be an ordered field. The set $\mathcal X(K)$ of all orders of $K$  can be made into a topological space  by using as a subbase the family of Harrison sets of the form $$H_K(a):=\{P\in \mathcal X(K):\, a\in P\},\,\,\,a\in \dot K = K\setminus\{0\}.$$
It is known that the space $\mathcal X(K)$ is a Boolean space (i.e. compact, Hausdorff and totally disconnected).

If $P$ is an order of $K$, then  the set 
$$A(P):= \{a\in K: \exists _ {q \,\in\,\mathbb Q^+}\,\, q \pm a\in P\}$$ 
is a valuation ring of $K$ with the maximal ideal 
$${I(P)}:= \{a\in K: \forall_ {q \,\in\,\mathbb Q^+}\,\, q\pm a\in P\}.$$ 
Moreover, the residue field $k(P) = A(P)/I(P)$ is ordered by an Archimedean order
$$\bar P := (P\cap \dot A(P)) + I(P),$$ where  $\dot A(P)$ is the set of units of $A(P)$. 

The set  $$\mathcal H(K)=  \bigcap_{P \in \mathcal X(K)} A(P )$$ is called {\em the real holomorphy ring} of the field $K$.

Since $k(P)$  has an Archimedean order, we consider it as a subfield of $\mathbb R$ and the place associated to $A(P)$ is called an {\em $\mathbb R$ - place}. Moreover, every place of $K$ with values in $\mathbb R$  is determined by some order of $K$. We denote by   $M(K)$ the set of all $\mathbb R$ - places of the field $K$. In fact, we have a surjective map
$$\lambda_K: \mathcal X(K)\longrightarrow  M(K).$$
and we can equip  $M(K)$ with the quotient topology inherited  from $ \mathcal X(K)$.

Note, that in terminology of  $\mathbb R$ - places we have $$\mathcal H(K)= \{a \in K:\;\; \forall_ { \xi\,\in\,M(K)}:\, \xi(a)\neq \infty\}.$$ For any $a \in \mathcal H(K)$ the map 
$$e_a:\,\, M(K)\longrightarrow\mathbb R, \;\;\;\;\;e_a(\xi) = \xi(a)$$ is called the evaluation map. The evaluation maps are continuous in the quotient topology of $M(K)$ and the set of evaluation maps separates  points of $M(K)$. Therefore $M(K)$ is a Hausdorff space. It is also compact as a continuous image of a compact space. By \cite [Theorem 9.11]{l},  the family of sets $$U_K(a) = \{\xi\in M(K):\, \xi(a)>0\},\,\,\text{for} \,\,a\in \mathcal H(K).$$ is a subbase for the topology of $M(K)$.
These sets may not  be closed and therefore the space $M(K)$ need not be  Boolean. However, every Boolean space is realized as a space of $\mathbb R$ - places of some formally real field. For this result see \cite {o}.

On the other hand there are a lot of examples of fields for which the space of  $\mathbb R$ - places has a finite number of connected components and even turns out to be connected. If  $K$ is a real closed field, then the space $M(K)$ has only one point. It is well known (see   \cite{sch}, \cite{bg}) that the space of $M(\mathbb R(X))$ is homeomorphic to a circle ring. 
More general result states that the space of  $\mathbb R$ - places  of a rational function field $K(X)$ is connected if and only if $M(K)$ is connected (see \cite{sch}, \cite{h}).

The goal of this paper is to describe the space  of $\mathbb R$ - places of a field $R(X)$,
where  $R$ is any real closed field. The set of orders of $R(X)$ is in one - to - one corespondence with Dedekind cuts of $R(X)$. Therefore the set $\mathcal X(R(X))$ is linearly  ordered. 
The main theorem of the second section shows how the map $\lambda: \mathcal X(R(X)) \longrightarrow
M(R(X))$ glues the points. The natural application is given for {\em the continuous closure} of the field $R$.
 This is a field $\tilde R \supseteq R$ such that $R$ is dense in $\tilde R$ and which is maximal in this respect.
We shall show that the spaces $M(R(X))$ and $M(\tilde R(X))$ are homeomorphic.
Using this result and modifying a method of \cite{z} we describe in the third section connection between completeness (in the sense of uniformity) and \textit{continuity} (see \cite{ba}) of an ordered field endowed with valuation topology.
These results allow us to describe 
the space of $\mathbb R$- places of the field $\textbf{R}(Y)$, where $\textbf{R}$ is a fixed real closure of $\mathbb R(X)$. In the fourth section we will show that it is not metrizable. The space $M(\textbf{R}(Y))$  is known to be subspace $M(\mathbb R(X,Y))$. Therefore $M(\mathbb R(X,Y))$ can not be a metric space.

\section {The real places of $R(X)$}

Let $R$ be a real closed field with its unique order $\dot R^2$. Denote by $v$ the valuation associated to $A(\dot R^2)$. Suppose that $\Gamma$ is the value group of $v$.

By \cite{g}, \cite{z} there is one-to-one correspondence between orders of $R(X)$ and Dedekind cuts of $R$.
If $P$ is an order of  $R(X)$, then the corresponding cut $(A_P, B_P)$ is given by  $A_P =\{a \in R:\,a<X\}$  and $B_P =\{b \in R:\,b>X\}$. On the other hand, if $(A, B)$ in any cut in $R$, then a set
$$P = \{ f \in R(X): \exists_{a\in A} \exists_{b \in B} \forall_{c \in (a,b)} \,f(c) \in \dot R^2\}$$
is an order of $R(X)$ which determines the cut $(A,B)$.

The cuts   $(\emptyset, R)$ and $(R, \emptyset)$ are called {\em the improper cuts}. The orders determined by these cuts are denoted $P_{\infty}^-$ and $P_{\infty}^+$, respectively. 
A cut $(A,B)$ of $R$ is called {\em  normal} if it satisfies a following condition 
\begin{displaymath}
 \forall_{c \in \dot R^2}  \exists_{a\in A} \exists_{b \in B}\,\,\,\,\,b-a<c.
\end{displaymath}

If  $A$ has a maximal or $B$ has a minimal element, then  $(A, B)$ is {\em a principal cut}.
Every $a \in R$ defines two principal cuts with the corresponding orders $P_a^-$ and $P_a^+$.
Note that if $R$ is a real closed subfield of $\mathbb R$, then  all of proper cuts of $R$ are normal. Moreover, if $R = \mathbb R$, then all of proper cuts  are normal and principal.

If  $A$ has not  a maximal element and $B$ has not a minimal element, then we say that $(A, B)$ is {\em a free cut} or {\em a gap}. If $R$ is not contained in $\mathbb R$, then $R$ has the \textit{ abnormal } gaps, i.e. gaps which are not normal.  For example, $(A, B)$ where
$$A = \{ a \in R:\,\, a\leqslant 0\}\cup \{a \in R:\,\, a> 0 \,\,\wedge\,\, a \in I(\dot R^2)\} \,\,\text{and}$$
$$B = \{b \in R:\,\, b> 0 \,\,\wedge\,\, b \notin I(\dot R^2)\},$$
is an abnormal gap in $R$.

In fact we can have three kinds of proper cuts:\\
(1) principal cuts;\\
(2) normal (but not principal) gaps;\\
(3) abnormal gaps.

One has to note that the correspondence between cuts in $R$ and orders of $R(X)$ makes the set $\mathcal X(R(X))$ linearly ordered. If $P$ and  $Q$ are different orders of  $R(X)$, then we say that
$$P\prec Q \Longleftrightarrow A_P\subset A_Q.$$
Let $P_a^- \prec P_a^+$ be the orders corresponding to principal cuts given by $a \in R$ Observe that the interval 
$(P_a^- , P_a^+)$ is empty. In such situations we shall say that  $\prec$ has a step in $a$.
The map
$$\lambda: \mathcal X(R(X))\longrightarrow M(R(X))$$
glues these steps.
By \cite{h}, $M(R(X))$ is connected space. However it can happen that $\lambda$ glues more points of $\mathcal X(R(X))$. Our goal is to answer the question: {\em Which points of $\mathcal X(R(X))$ does $\lambda$ glue?} More exactly, suppose that $(A_1, B_1)$ and $(A_2, B_2)$ are the cuts corresponding to the orders $P_1$ and $P_2$, respectively.
{\em When $\lambda(P_1) = \lambda(P_2)$?}
We shall make use of Separation Criterion \cite[Proposition 9.13]{l}, which allows to separate $\mathbb R$ - places.
\begin{theorem}{\bf [Separation Criterion]}\label{sc}\\
 Let $P$ and $Q$ be distinct orders of a field $K$. Then $\lambda_K(P) \neq \lambda_K(Q)$ if and only if there exists $a \in K$ such that $a \in \dot A(P)\cap P$ and $-a \in Q$.
\end{theorem}

Let $P$ be an order of $R(X)$. Then 
$$A(P) = \{ f \in R(X):  \exists_{q\in \mathbb Q^+} \exists_{a\in A} \exists_{b \in B} \forall_{c \in (a,b)} \, q \pm f(c) \in \dot R^2\}=$$ 
$$\{ f \in R(X): \exists_{a\in A} \exists_{b \in B} \forall_{c \in (a,b)} \,f(c) \in A(\dot R^2)\}, \text{  and  }$$
$$I(P) = \{ f \in R(X):  \forall_{q\in \mathbb Q^+} \exists_{a\in A} \exists_{b \in B} \forall_{c \in (a,b)} \, q \pm f(c) \in \dot R^2\}.$$ 
By {\em a neighborhood} of a cut $(A,B)$ we mean an interval $(a,b)\subset R$ such that $(a,b)\cap A \neq \emptyset$ and $(a,b)\cap B \neq \emptyset$.

\begin{remark}\label{v}
 Let $P$ be an order of $R(X)$. Then $A(P)$ is a set of these functions which on some neighborhood of  $(A_P,B_P)$ have values in  $A(\dot R^2)$ and $\dot A(P)$ contains these functions which on some neighborhood of  $(A_P,B_P)$ have values in  $\dot A(\dot R^2)$.
\end{remark}
According to \cite[Lemma 2.2.1]{z}, every  cut  of $R$ determines a lower cut
$$S = \{v(b-a):\,\,a \in A,\, b\in B\},$$
 in $\Gamma$.
Note that if $(A, B)$ is a normal cut, then $S= \Gamma$ and if $(A, B)$ is an improper cut, then $S= \emptyset$. The sets $S$ allow us to compare gaps.
We can say that a gap $(A_1,B_1)$ is "bigger" then   $(A_2,B_2)$ if  $S_1\subset S_2$. In a similar way one can 
 investigate a "distance" between two cuts $(A_1,B_1)$ and  $(A_2,B_2)$.  Suppose that $A_1 \subset A_2$ and consider the set
$$U = \{v(a'-a):\,\,a, a' \in B_1 \cap A_2, \, \, a<a'\}.$$ Suppose that $S_1$ and $S_2$ are  the lower cuts in $\Gamma$ determined by $(A_1,B_1)$ and  $(A_2,B_2)$.  
\begin{lemma}\label{u}
 $U$ is an upper cut in $\Gamma$. Moreover, $\Gamma \setminus (S_1\cap S_2) \subset U$.
\end{lemma}
\begin{proof}
 We shall show that if $\gamma \in U$ and  $\gamma' \in \Gamma$ and $\gamma< \gamma'$, then $\gamma' \in U$.
We have $\gamma  = v(a'-a)$, where $a$ and $a'$ are as in the definition of $U$ and $\gamma' = v(c)$, where $c$ is a positive element of $R$. Since $v(a'-a) < v(c)$, we have $v(\frac{c}{a' - a}) \in I(\dot R^2)$. Then $\frac{c}{a' - a} <1$.
So $a+c <a'$. Thus $a+c \in B_1 \cap A_2$. We have $v(a+c\,-\,a) = v(c) = \gamma' \in U$.

Now suppose that  $\gamma \in \Gamma \setminus (S_1\cap S_2) $. Let $c$ be a positive element of $R$ with $v(c) = \gamma$. Fix an element $a \in B_1 \cap A_2$. Assume that  $\gamma  \notin S_1$. Then $a - c \in  B_1 \cap A_2$. Thus $v(a\,-\, (a-c)) = v(c) = \gamma \in U$. If $\gamma  \notin S_2$, then $a + c \in  B_1 \cap A_2$ and $v(a+c\,-\, a) = v(c) = \gamma \in U$.
\end{proof}
\begin{theorem}\label{rx}
 Let  $(A_1,B_1)$ and  $(A_2,B_2)$ be the cuts in $R$ corresponding to  the orders $P_1$ and $P_2$ of $R(X)$, respectively.  Let $S_1$, $S_2$ and $U$ be the cuts in $\Gamma$ defined above and let $\lambda:\mathcal X(R(X))\rightarrow M(R(X))$ be the canonical map. 
\begin{enumerate}
 \item If $S_1 \neq S_2$, then $\lambda(P_1)  \neq \lambda(P_2)$.
\item  If $S_1 = S_2 = S$ and $S\cap U \neq \emptyset$, then $\lambda(P_1)  \neq \lambda(P_2)$.
\item  If $S_1 = S_2 = S$ and $S\cap U = \emptyset$,  then $\lambda(P_1) = \lambda(P_2)$.
\end{enumerate}
\end{theorem}
\begin{proof} Without lost of generality we can assume that $A_1\subset A_2$.\\
 (1) Suppose  that $S_1 \subset S_2$. Then there exist $a \in B_1\cap A_2$ and $b \in B_2$ such that $v(b-a) \notin S_1$. Consider a linear polynomial $f(X) = \frac{X-a}{b-a}+1$. This polynomial has a root  $x_0 = a - (b-a)$. If $x_0 \in A_1$, then $v(a -x_0) = v(b-a) \in S_1$  - a contradiction. Therefore $x_0 \in B_1$.  Moreover,
$f(a) = 1$ and $f(b) = 2$. Therefore $f$ has positive values in some neighbourhood of  $(A_2,B_2)$ which are units in $A(\dot R^2)$ and negative values in some neighbourhood of  $(A_1,B_1)$. By Remark \ref{v} and Separation Criterion 
 $\lambda(P_1)  \neq \lambda(P_2)$. If  $S_2 \subset S_1$, then $f$ can be defined as a suitable, decreasing linear polynomial.\\
(2) Let $\gamma \in U \cap S$. Then there exist $a\in A_1$, $b \in B_1$ and $c,d \in B_1 \cap A_2$ such that $\gamma =  v(b-a) = v(d-c)$. 

 We shall show that one can fix $\gamma$ in  such a way that $ a<b\leqslant c<d$. If $c<b$, then  we take $\gamma'= v(c-a)$. We have $\gamma '\geqslant \gamma$, so $\gamma' \in U$. If $\gamma' > \gamma$, then  $v(c-a) >v(d-c)$.
Thus
$c-a <d-c$ and $c +(c-a)< d$. Therefore $c +(c-a) \in B_1 \cap A_2$. So we can take $\gamma '$ as $\gamma$, $c$ as  $b$ and $c +(c-a)$ as  $d$. 

Since   $v(b-a) = v(d-c)$, there exists $n \in \mathbb N$ such that $\frac{1}{n} <\frac{b-a}{d-c}< n$. Then $\frac{b-a}{n} < d-c$. Consider a linear polynomial $f(X)$ such that $f(a) = n+1$ and $f(b) = 1$ i.e. $f(X) = \frac{n(b-X)}{b-a} +1$. This polynomial has a root  $x_0 = b + \frac{b-a}{n}< b+ d-c <d$. Thus $f$ has positive values in a neighborhood of  $(A_1,B_1)$ which are units in $A(\dot R^2)$ and negative values in a neighborhood of  $(A_2,B_2)$. Using Separation Criterion we get  $\lambda(P_1)  \neq \lambda(P_2)$.\\
(3) By \cite [Corollary 2,13 ]{l} it suffices to show that $\dot A(P_1) \cap P_1 = \dot A(P_2) \cap P_2.$\\
(i) From Lemma \ref{u} we have $U = \Gamma \setminus S$.\\
(ii) We can assume that $0 \in  B_1\cap A_2$. 

If $a \in B_1\cap A_2$, then we can consider the cuts: $(A_1 -a,  B_1 -a)$ and $(A_2 -a,  B_2 -a)$. Then  $f(X)$ is a "separation" function for $(A_1,  B_1 )$ and $(A_2,  B_2)$ if and only if $f(X+a)$ is a "separation" function for 
 $(A_1 -a,  B_1 -a)$ and $(A_2 -a,  B_2 -a)$. Therefore  the $\mathbb R$- places determined by orders associated to 
$(A_1 -a,  B_1 -a)$ and $(A_2 -a,  B_2 -a)$ are equal if and only if $\lambda(P_1) = \lambda(P_2)$.\\
(iii) $(A_1,  B_1 )$ and $(A_2,  B_2)$  are symmetric in respect to 0 i.e. if $a \in B_1\cap A_2$, then  $-a \in B_1\cap A_2$.

If  $a \in B_1\cap A_2$ and $-a \in A_1$, then  $S \ni v(a) = v(-a) \in U$ - a contradiction with  $S\cap U = \emptyset$.\\
(iv) Let $A:= B_1\cap A_2$ and $B = A_1 \cup B_2$. Then $v(A) = U \cup\{\infty\}$ and $v(B) = S$.
Let $v_+$ be a valuation determined by order $P_1$ and let $v_-$ be a valuation determined by order $P_2$.
The valuations $v_1$ and $v_2$ are extensions of $v$. We have
$$v(a) > v(b),\,\,\text{ for }\,\, a \in A, b\in B.$$
Therefore
$$v(a) \geqslant v_{\pm}(X) \geqslant v(b),\,\,\text{for}\,\, a \in A, b\in B.$$
(v) A valuation group of $v_{\pm}$ is bigger then $\Gamma$. In fact
$$v(a) > v_{\pm}(X) > v(b),\,\,\text{for}\,\, a \in A, b\in B.$$

If $v_{\pm}(X) = v(a)$ for some $a \in A$ then $v_{\pm}(\frac{X}{a} )= 0$. Thus the function $\frac{X}{a}$ has Archimedean values in some neighbourhood of $(A_1,  B_1 )$.   Therefore there exists $b \in B$ such that $v(\frac{b}{a}) = 0$. So $v(b) = v(a)$ - a contradiction. Similarly $v_{\pm}(X) \neq v(b)$.\\
(vi) We shall check the values of valuations $v_{\pm}$ on linear and quadratic polynomials:\\
$v_{\pm}(X - a) = v_{\pm}(X)$ for $a \in A$;\\
$v_{\pm}(X - b) = v(b)$ for $b \in A$;\\
$v_{\pm}((X - c)^2 + d^2) =
\left\{ \begin {array}{l}
2v_{\pm}(X)\text{ if } c,d \in A\\
2v(d)\text{ if } c \in A,d \in B\\
2v(c)\text{ if } c \in B,d \in A\\
2\min\{v(c),v(d)\}\text{ if } c, d\in B\\
\end {array}\right.$.\\
(vii) Since $\Gamma $ is a divisible group, $n\cdot v_{\pm}(X ) \notin \Gamma$ for $n \in \mathbb N$.
Thus for every polynomial $f \in R[X]$, there exist $a\in R$ and $n \in \mathbb N$ such that
 $v_{\pm}(f ) = v(a) + n\cdot v_{\pm}(X)$ and $v_{\pm}(f ) = 0$ if and only if  $v(a) = 0$ and $n=0$. Therefore $\dot A(P_1) = \dot A(P_2).$ \\
(vii) Let $F \in \dot A(P_1)\cap P_1$. Then $F$ has a representation:
$$F = \frac{f_1}{f_2} = \frac{c(X-c_1)\cdot....(X-c_k)\cdot[\text{a product of sum of squares}]}{d(X-d_1)\cdot....(X-d_k)\cdot[\text{a product of sum of squares}]}$$
Since $v_{\pm}(F) = 0$, 
$$\sharp \{c_i \in A: i=1,...k\}-\sharp \{d_i \in A: i=1,...l\}\equiv 0\,\,(\text{mod }2).$$
Thus $f_1\cdot f_2$ has an even number of roots in $A$.   Therefore $F \in  P_2$.
\end{proof}
\begin {remark}
The Theorem \ref{rx} shows that the map $\lambda$ glues only the abnormal  gaps which are close one  to each other. For example, the orders determined by gaps $(A_1,B_1)$ and $(A_2,B_2)$ where\\
$A_1=\{a\in  -R^2 : a\notin I(\dot R^2)\}$, $B_1=\{a\in  -R^2 : a\in I(\dot R^2)\}\cup R^2$,\\
$A_2= \{-R^2 \cup\{a\in  R^2 : a\in I(\dot R^2)\}$, $B_2=\{a\in  R^2 : a\notin I(\dot R^2)\}$,\\
are always  glued, since then  $S_i = \{\gamma \in  \Gamma:\;\; \gamma \leq 0\}$ for $i=1,2$ and\\ $U = \{\gamma \in  \Gamma:\;\; \gamma > 0\}$.
\end{remark}
\begin{remark}
 The orders $P_a^+$ and $P_a^-$ determine the same $\mathbb R$-place, since $U = \emptyset$. Also  $P_{\infty}^+$ and $P_{\infty}^-$ determine the same $\mathbb R$-place, since $S = \emptyset$.
\end{remark}

\begin {remark}
 If $R$ is a real closed subfield of $\mathbb R$, then every cut of $R$ is normal. By Theorem \ref{rx} the space $M(R(X))$ is homeomorphic to $M(\mathbb R(X))$.
\end {remark}
The above remark can be easily generalized if one replace $R$ and $\mathbb R$ by any real closed field $K$ and its continuous closure $\tilde R$. Let us recall a definition from \cite{ba}.

\begin{definition} 
The ordered field $K$ is called {\em continuous closed} if every normal cut in $K$ is principal.
We say that an ordered field $\tilde K$ is {\em a continuous closure of $K$} if $\tilde K$ is continuous closed and $K$ is dense in $\tilde K$.
\end{definition}

The continuous closure $\tilde K$ is uniquely determined for every ordered field $K$. Moreover, if $K$ is real closed, then $\tilde K$ is also real closed (see \cite{ba}).

\begin{theorem}\label{rc}
 Let $R$ be a real closed field and let  $\tilde R$  be its  continuous closure. Then spaces $M(R(X))$ and $M(\tilde R(X))$ are homeomorphic. 
\end{theorem}
\begin{proof}
 By \cite{o2}, the restriction map $\omega:M(\tilde R(X))\longrightarrow M(R(X))$, $\;\;\omega(\xi) = \xi \mid _{R(X)}$ is continuous. It suffices to show that it is a bijection.

Fix $\xi \in M(R(X))$. There is $P \in \mathcal X(R(X))$ such that $\lambda_{R(X)}(P)=\xi$. Let $(A_P,B_P)$ be a cut in $R$ corresponding to $P$. Let $\tilde A = \{\tilde a \in \tilde R: \exists_{a \in A_P} \;\; \tilde a < a\}$ and
let $\tilde B$ be the completion of $\tilde A $ in $\tilde R$.
 If $(\tilde A, \tilde B)$ is a cut (i.e. if $(A_P,B_P)$ is principial or abnormal gap), then the order $\tilde P$ corresponding to $(\tilde A, \tilde B)$ and restricted to $R(X)$ is equal to $P$. 
Otherwise, there exists a unique element $\tilde c \in \tilde R$ such that 
$$\forall_{a \in \tilde A}\;\; \forall_{b \in \tilde B}\;\;a < \tilde c < b.$$ Then $(\tilde A \cup \{\tilde c\}, \tilde B)$ is a cut in $\tilde R$. Let $\tilde P$ be an order corresponding to this cut. Observe, that $\tilde P \cap R(X) = P$.  In fact, if $f \in R(X)$ is positive on some neighborhood of $(\tilde A \cup \{\tilde c\}, \tilde B)$, then
it takes positive values on some neighborhood of $( A_P, B_P)$ (note that $\tilde c$ is not algebraic over $R$ since $R$ is real closed). Then in both cases, $\lambda_{\tilde R(X)}(\tilde P) \mid_{R(X)} = \xi$. Thus $\omega$ is surjective.

To show injectivity of $\omega$ it suffices to observe the following. Let $\tilde \Gamma$ be a value group of valuation $v$ corresponding to the unique ordering of $\tilde R$. Let $\tilde A \subset \tilde \Gamma$. Assume that $v(\tilde c) \in \tilde A$ for some $\tilde c \in \tilde R$. By density of $R$ in $\tilde R$ in any sufficiently small neighborhood of $\tilde c$ one can find $c \in R$ such that $v(c) = v(\tilde c)$. Thus $\{v(c): \;\;c \in R\} \cap \tilde A = \tilde A$.  Now one can use Theorem \ref{rx} (i) and (ii). 
\end{proof} 

\section {Completeness and continuity}

\begin {notation}
 Let $k$ be a field, $\Gamma $ be a linearly ordered abelian group. The field of (generalized)  power series $k((\Gamma))$ is the set of formal series \begin{displaymath}
a = \sum_{\gamma \in \Gamma} a_{\gamma} x^{\gamma}
\end{displaymath} with well - ordered support $$supp(a) = \{ \gamma:\;\; a_{\gamma} \neq 0\}.$$
Sum and multiplication  are defined as follow:
\begin{displaymath}
 a +b = \sum_{\gamma \in \Gamma} (a_{\gamma}+b_{\gamma}) x^{\gamma};\;\;\;\;\;
ab = \sum_{\gamma \in \Gamma}(\sum_{\delta+\eta = \gamma} a_{\delta}b_{\eta}) x^{\gamma}
\end{displaymath}
The fact that $k((\Gamma))$ is a field was shown by Hahn \cite{hah}. This field is ordered by lexicographic order. The natural valuation $v: k((\Gamma)) \longrightarrow \Gamma \cup\{\infty\}$ corresponding to this order is given by formula
$$v(a) = min\, supp(a).$$

Let $R$ be real closed field, $v: R\longrightarrow \Gamma \cup\{\infty\}$ be the natural valuation of $R$ with the Archimedean residue field $k$. The proof of the following theorem one can find in \cite{z}, compare\cite{kap}.
\begin{theorem}
 There exists a field embedding $R\hookrightarrow k((\Gamma))$ preserving the order and  valuation.
\end{theorem}

\end{notation}

In further part of this chapter we will use the notion of uniform spaces (see \cite{e}, ch.8) and related notions: base of the uniformity (\cite{e}, Chapter 8.1); completeness, Cauchy and convergent nets (\cite{e}, Chapter 8.3). 
 
One of examples of uniform spaces is a field $K$ with a valuation $v$:  the base of uniformity is the family $\{V_{\gamma}:\;\;\gamma \in \Gamma\}$ where   
 $$V_{\gamma} = \{(a,b) \in K \times K:\;\;v(a - b)>  \gamma\}$$ for every $\gamma \in \Gamma$, and $\Gamma$ is a value group of $v$. 

Let us recall definitions of uniform notions in particular case of a field $K$ (with valuation $v$ and value group $\Gamma$).

\begin{definition} $K$ is \textit{complete} if every centered family of closed sets, which contains arbitrarily small sets, has one-point intersection. A family $\mathcal{F}$ contains \textit{arbitrarily small sets} if
$$\forall_{\gamma \in \Gamma} \exists_{F\in \mathcal{F}} \forall_{x,y\in F} \ v(x-y) > \gamma$$

\end{definition}

\begin{definition}
We say that a net $(a^{\sigma} : \sigma \in \Sigma )$ is a Cauchy set if
$$\forall_{\gamma\in \Gamma} \exists_{\sigma_0 \in \Sigma} \forall_{\sigma > \sigma_0} \ v(a^{\sigma} - a^{\sigma_0} ) > \gamma $$
Similarly, we say that a net $(a^{\sigma} : \sigma \in \Sigma )$ is convergent to $a\in K$ if 
$$\forall_{\gamma\in \Gamma} \exists_{\sigma_0 \in \Sigma} \forall_{\sigma > \sigma_0} \ v(a^{\sigma} - a ) > \gamma $$
\end{definition}

\begin{theorem}[\cite{e}, Theorem 8.3.20] 
A uniform space $X$ is complete if and only if every Cauchy net in $X$ is convergent. 
\end{theorem}

Let $k((\Gamma))$ be formal power series field, where $k$ is Archimedean ordered. Let $\kappa$ be cardinal number which is cofinality of $\Gamma$. Consider a subfield\\
\begin{displaymath}
 k_{\kappa}((\Gamma)) = \{a \in k((\Gamma));\;\;\; \mid supp(a) \mid < \kappa \text{ or } supp(a)\text{ is cofinal in  } \Gamma \text { of order type } \kappa\}.
\end{displaymath}
\begin{claim}
$ k_{\kappa}((\Gamma))$ is complete.
\end{claim}
\begin{proof}
 Let $(a^{\sigma}: \sigma \in \Sigma )$ be a Cauchy net, where 
$a^{\sigma} = \sum_{\gamma \in \Gamma} a^{\sigma}_{\gamma} x^{\gamma}$ and $\Sigma$ is directed set.

Fix cofinal sequence $(\gamma_{\delta} : \delta < \kappa )$ in $\Gamma$. 
Since $(a^{\sigma}: \sigma \in \Sigma )$ is a Cauchy net, we have
$$ \forall_{ \delta <\kappa} \exists_{\sigma_{\delta}} \forall_{\tau > \sigma_{\delta}} \ 
v(a^{\sigma_{\delta}} - a^{\tau} ) \geq \gamma_{\delta +1} $$
The condition  $v(a^{\sigma_{\delta}}- a^{\tau} ) \geq \gamma_{\delta +1}$ means that 
$a^{\sigma_{\delta}}_{q} =  a^{\tau}_{q} $  for every $q< \gamma_{\delta +1}$.
Put $$a^{\infty}_q = 
\begin{cases}
a^{\sigma_{\delta}}_{q} &   q\in [ \gamma_{\delta} , \gamma_{\delta +1} ) \\
a^{\sigma_{0}}_{q} &   q <   \gamma_{1}
\end{cases}.
$$
It is easy to see that $a^{\infty} = \sum_{\gamma \in \Gamma} a^{\infty}_{\gamma} x^{\gamma}$ is a limit of given net.
In fact, it is sufficient to observe that $a^{\infty}_q = a^{\sigma_{\delta} }_q  = a^{\tau}_q $ for every $q< \gamma_{\delta +1}$.
\end{proof}

Now we will show how completeness of an ordered field implies its continuity.

\begin{theorem}\label{con}
Let $(K, P)$ be an ordered field with valuation $v$ determined by   $P$. 
If $K$ is complete, then $K$ is continuous.
\end{theorem}
\begin{proof}
Consider a normal Dedekind cut $(A,B)$ in $K$.  We will show that it is a principal cut. 
By definition of a normal cut  for every $\gamma \in \Gamma $ there are $a_{\gamma} \in A$ and $b_{\gamma} \in B$ such that $v(a_{\gamma} - b_{\gamma}) >\gamma$.
The intervals $[a_{\gamma}, b_{\gamma} ]$ are closed in uniform topology and $v(x-y)>\gamma $ for any $x,y \in [a_{\gamma}, b_{\gamma} ]$. Moreover, the family 
$\mathcal{F} = \{ [a_{\gamma}, b_{\gamma} ] : \gamma \in \Gamma \}$ is centered.
By a definition of completeness there is a unique $c\in K$ such that 
$$\bigcap \mathcal{F} = \bigcap_{\gamma \in \Gamma} [a_{\gamma}, b_{\gamma} ] = \{ c\} .$$
 To prove that $a\leq c$ for every $a\in A$  assume the opposite: $c<a$ for some $a\in A$. 
Let $\gamma_0=v(c-a)$.
 There is $b\in B$ such that $v(a-b) > \gamma_0$. 
 Since a family $\mathcal{F}\cup \{ [a,b] \}$ is centered  we have also
$$ c\in \bigcap (\mathcal{F} \cup \{ [a,b] \} ) \subset [a,b],$$
which is a contradiction. 
Similarly one can prove that $c\leq b$ for every $b\in B$. 
\end{proof}

\begin{example}\label{ex}
 Let $\textbf{R}$ be a fixed real closure of $\mathbb R(X)$. Then $\Gamma = \mathbb Q$. Let
$\mathbb R(\mathbb Q)\subset \mathbb R((\mathbb Q)) $  be the set of formal power series with finite support.
We have 
$$\mathbb R(\mathbb Q)\subset\textbf{R}\subset  \mathbb R_{\aleph_0}((\mathbb Q)). $$
By Theorem \ref{con}, $\mathbb R_{\aleph_0}((\mathbb Q))$ is continuous.  Note that $\mathbb R(\mathbb Q)$ is dense in 
$\mathbb R_{\aleph_0}((\mathbb Q))$. Thus $\mathbb R_{\aleph_0}((\mathbb Q))$ is the continuous closure of  $\textbf{R}$.
\end{example}

\section{The space $M(\textbf{R}(Y))$ }
To describe the topology of the space $M(\textbf{R}(Y))$, where $\textbf{R}$ is a fixed real closure of $\mathbb R(X)$, we shall modify some methods of \cite{z}. 

Let $R$ be any real closed field and let $P$ be an order of the field $R(Y)$ determining the valuation $v$ with value group $\Gamma$ and the Archimedean residue field $k$.  Let $(A_P, B_P)$ be a cut in $R$ corresponding to $P$ and let $S$  be the lower cut in $\Gamma$ determined by $(A_P, B_P)$ (see section 2).

A formal series $\tilde p =  \sum_{\gamma \in \Gamma} \tilde p_{\gamma} x^{\gamma} \in k((\Gamma))$ is defined in the following way (see \cite {z}):
\begin{enumerate}
 \item If $\gamma \notin S$ then  $\tilde p_{\gamma} =  0$.
\item Suppose that $\gamma \in S$ but $\gamma$ is not a maximal element in $S$. Then there exists $a \in A_P$ and $b \in B_P$ such that $v(b-a)> \gamma$. Then   $\tilde p_{\gamma} =  a_{\gamma} = b_{\gamma}$ ($\tilde p_{\gamma}$ does not depend from the choise of $a$ and $b$ - see \cite{z}, Prop.2.2.3).
\item Suppose that $\gamma $  is a maximal element in $S$. Let $M = \{ a_{\gamma}  \in A_P;\,\, \exists_{b\in B_P} \;v(b-a) > \gamma\}$ and $N = \{ b_{\gamma}  \in B_P;\,\, \exists_{a\in A_P} \;v(b-a) > \gamma\}$.  Then there exists exactly one $r \in \mathbb R$ such that $M \leq r \leq N$ (see \cite{z}, Prop.2.2.3).  Then  $\tilde p_{\gamma} = r$.
\end{enumerate}

\begin{remark}
A series $\tilde p$ does not determine an order uniquely.
For example, if $R$ is non - Archimedean real closed field, then formal series $\tilde p$ constantly equal to 0 correspond to orders given by following cuts:
\begin{itemize}
 \item the principal cuts in 0;
\item gaps between infinitely small and other Archimedean elements;
\item gaps between Archimedean and infinitely large  elements.
\end{itemize}
\end{remark}

To distinguish orders one has to consider the lower cuts $S$ and a sign defined below. Consider three cases:
\\
(+) There exists $a \in A_P$ such that $v(a - \tilde p) \notin S.$ Then take  a symbol $(S, \tilde p, +)$\\
(-) There exists $b \in B_P$ such that $v(b - \tilde p) \notin S.$ Then take  a symbol $(S, \tilde p, -)$\\
(.) For every  $c \in R$, $v(c - \tilde p) \in S.$ Then take  a symbol $(S, \tilde p)$
\begin{remark}
 In the case R = $\mathbb R_{\aleph_0}((\mathbb Q))$ the last case does not hold, because $\tilde p \in R$.
\end{remark}

\begin{remark}\label{zek}(\cite{z} p.33)
 The element $\tilde p$ defined above has properties:
\begin{enumerate}
 \item $\forall_{\gamma \in S} \exists_{a\in R}\;\; v(a - \tilde p )\geq \gamma$;
\item $\forall_{\gamma \notin S} \;\;\tilde p_{\gamma} = 0$.
\end{enumerate}
\end{remark}
\begin{theorem}[\cite{z}, Theorem 2.2.6]\label{os}
 There is one - to - one correspondence between orders of the field $R(Y)$ and the symbols $(S, \tilde p, +)$, $(S, \tilde p, -)$ and $(S, \tilde p)$  for  $\tilde p \in \mathbb R((\Gamma))$ satisfying conditions (1), (2) of Remark \ref{zek}.
\end{theorem}
Now we shall restrict to the case when $R$ is a fixed real closure of $\mathbb R(X)$. 
Let $((X))$ be a set of series of the form $p = \sum_{} p_\gamma x^{\gamma}$, where $\gamma \in \mathbb Q \cup \{\infty\}$,\;\; $p_{\gamma} \in \mathbb  R \cup \{\pm \infty\}$ having following properties:
\begin{enumerate}
 \item $p_{\infty} \in \{\pm \infty\};$
\item the support $supp (p) = \{\gamma:\,\, p_{\gamma} \notin \{0, \pm \infty\}\}$ is finite or cofinal in $\mathbb Q$ of order type $\aleph_0$;
\item if $p_{\gamma} = \varepsilon\infty $, then  $p_{\delta} = \varepsilon\infty$, for every $\delta> \gamma$ and $\varepsilon \in \{+,-\}$.
\end{enumerate}
Note that the map $(S,\tilde p , \varepsilon) \longmapsto p \in ((X))$, where
 $$p_{\gamma} = \left\lbrace \begin{array}{l}
                 \tilde p_{\gamma}, \text{ for } \gamma \in S\\
\varepsilon\infty \text{ for } \gamma \notin S
                           \end{array}\right.$$
is a bijection. This fact and Theorem \ref{os} leads to following corollary.
\begin{corollary}
 There is one - to - one correspondence between the orders of the field $\mathbb R_{\aleph_0}((\mathbb Q))(Y)$ and $((X))$.
\end{corollary}
We consider the topology of lexicographic order on $((X))$.
\begin{proposition}
 $((X))$ is homeomorphic to $\mathbb R_{\aleph_0}((\mathbb Q))(Y)$.
\end{proposition}
\begin{proof}
Take the Harisson set $H(\frac{f}{g}) = H(fg) \subset \mathcal X(\mathbb R_{\aleph_0}((\mathbb Q))(Y))$. The polynomial $fg$ takes positive values of finite many intervals. An order $P$ belongs to $H(fg)$ if and only if corresponding cut $(A_P, B_P)$ has neighbourhood on which $fg$ takes positive values. So if $fg$ takes positive values on interval $(a,b)$ avery cut $(A_P, B_P)$ of $(a,b)$  gives an order $P$ which belongs to $H(fg)$. Additionally we should check what happens at the ends $a = \sum a_{\gamma}x^{\gamma}$ and $b =\sum a_{\gamma}x^{\gamma}$ of the interval $(a,b)$.
Let $a^-\prec a^+\prec b^-\prec b^+$ be a series from $((X))$ corresponding to the principal cuts in $a$ and $b$. Note that $a^+$ and $b^-$ belongs to $H(fg)$. Thus $(a^-, b^+) \subset H(fg)$. In other words $H(fg)$ is a finite sum of intervals  $(a^-, b^+) \subset ((X))$. such that $fg $ is positive on $(a,b)\subset \mathbb R_{\aleph_0}((\mathbb Q))$.
From the other side, note that an interval $(p,q)\subset ((X))$, can be replaced by the sum of Harrison sets $H(f)$, where $f$ is running through the all quadratic polynomials with roots $a,b \in \mathbb R_{\aleph_0}((\mathbb Q))$ such that $p \preceq a^+ \prec b^-\preceq q$ and positive on $(a,b)$.
\end{proof}
\begin{corollary}\label{m}
 The space of $\mathbb R$ - places of the field $\mathbb R_{\aleph_0}((\mathbb Q))(Y)$ is homeomorphic to the space
\begin{displaymath}
((M)) = \left\lbrace  \sum_{} p_\gamma x^{\gamma};\;\;\;\gamma \in \mathbb Q,\;\;\; p_{\gamma} \in \mathbb  R \cup \{ \infty\}\right\rbrace  
\end{displaymath}
of series satisfying following properties:
\begin{enumerate}
 \item $supp(p)$ is finite or cofinal in $\mathbb Q$ of order type $\aleph_0$;
\item if $p_{\gamma} = \infty$, then $p_{\delta} = \infty$ for every $\delta >\gamma$.
\end{enumerate}
with the quotient topology from $((X))$.
\end{corollary}
\begin{proof}
 By the construction of elements of $((X))$ and Theorem \ref{rx} we have that $ p = \sum_{} p_\gamma x^{\gamma}$,
$q = \sum_{} q_\gamma x^{\gamma}\,\, \in ((X))$ determine the same $\mathbb R$ - place if and only if 
 $p_{\gamma} = q_{\gamma}$ when one of them is a real number.
\end{proof}
\begin{lemma}
 The cellularity of  $((M))$ is not smaller then continuum $\mathfrak c $.
\end{lemma}
\begin{proof}
 We will define a family of parwise disjoint open sets of cardinality $\mathfrak c $. 
 
Let  $t \in \mathbb R$ and let $U_t$ be a set which contains all series  $c^t$ with properties:
\begin{itemize}
 \item[$ \cdot$] $ c^t_{\gamma} = 0 $ for $\gamma <1, \gamma \neq 0$
\item[$ \cdot$] $c^t_{0} = t $
\item[$ \cdot$] $c^t_{1} \in (-1,1)$
\end{itemize}
Observe that:
\begin{enumerate}
 \item $U_t$ is nonempty.
\item $U_t$ is open, because its inverse image in $((X))$ is an interval $(a^t, b^t)$, where
 \begin{tabular}{lcccl}
  $ a^t_{\gamma} = 
\left\lbrace  \begin{array}{l}
c^t_{\gamma} \text{ for } \gamma <1   \\
-1       \text{ for } \gamma = 1\\ 
+\infty   \text{ for } \gamma >1
\end{array} \right.$ & &&&$ b^t_{\gamma} = 
\left\lbrace  \begin{array}{l}
c^t_{\gamma} \text{ for } \gamma <1   \\
1       \text{ for } \gamma = 1\\
-\infty   \text{ for } \gamma >1 \end{array} \right..$ 
 \end{tabular}
\item  For $t<s$, $(a^t, b^t)<(a^s, b^s)$ and by using Theorem \ref{rx} we have that $U_t \cap U_s = \emptyset$.
\end{enumerate} \end{proof}
\begin{corollary}\label{met}
 $((M))$ is not metrizable.
\end{corollary}
\begin{proof}
 Since cellularity of any space is not greater than density, $((M))$ can not be separable. 
Since every compact, metric space is separable, $((M))$ can not be metrizable.
\end{proof}
Using Theorem \ref{rc}, Example \ref{ex}, Corollary \ref{m} and Corollary \ref{met} we get:
\begin{corollary}
Let $\textbf{R}$ be a fixed real closure of $\mathbb R(X)$. The space $M(\textbf{R}(Y))$ is not metrizable.
\end{corollary}
\begin{corollary}
 The space $M(\mathbb R(X,Y))$ is not metrizable.
\end{corollary}
\begin{proof}
Let $\textbf{R}$ be a fixed real closure of $\mathbb R(X)$.  $M(\textbf{R}(Y))$ is a subspace of
$M(\mathbb R(X,Y))$ (see \cite[Lemma 8]{cr}).
\end{proof}

\noindent
\textbf{Acknowledgment}. We thank Prof. A. S\l adek and Prof. M. Kula who spent time on reading this paper and gave us valuable comments.

\end{document}